\documentclass{amsart}
\usepackage{graphicx}
\usepackage{a4wide}
\usepackage{amssymb}
\usepackage{color}
\usepackage{stmaryrd}
\usepackage{tikz}
\newtheorem{theorem}{Theorem}[section]
\newtheorem{lemma}[theorem]{Lemma}
\newtheorem{proposition}[theorem]{Proposition}
\newtheorem{conjecture}[theorem]{Conjecture}
\newtheorem{cor}[theorem]{Corollary}
\theoremstyle{definition}
\newtheorem{definition}[theorem]{Definition}
\newtheorem{notation}[theorem]{Notation}
\newtheorem{remark}[theorem]{Remark}
\newtheorem{example}[theorem]{Example}
\newenvironment{theorem*}{\par\noindent\textbf{Box Theorem}\it}{\par\medskip}
\DeclareMathOperator{\ord}{ord}
\DeclareMathOperator{\corank}{corank}
\DeclareMathOperator{\Hom}{Hom}
\DeclareMathOperator{\im}{im}
\newcommand{\kk}{\mathsf k}
\newcommand{\ello}{l}

\newcommand{\kt}{\kk\llbracket t\rrbracket}
\title{Jordan Type stratification of spaces of commuting nilpotent matrices}
\author{Mats Boij}
\address{Department of Mathematics, KTH Royal Institute of Technology, SE-100 44 Stockholm, Sweden}\email{boij@kth.se}
\author{Anthony Iarrobino}
\address{Department of Mathematics, Northeastern University, Boston, MA 02115, USA}
\email{a.iarrobino@neu.edu}
\author{Leila Khatami}
\address{Department of Mathematics, Union College, Schenectady, NY 12308, USA}
\email{khatamil@union.edu}
\date{\today}
\begin{document}
\begin{abstract}
    An $n\times n$ nilpotent matrix $B$ is determined up to conjugacy by a partition $P_B$ of $n$, its \emph{Jordan type} given by the sizes of its Jordan blocks. The Jordan type $\mathfrak D(P)$ of a nilpotent matrix in the dense orbit  of the nilpotent commutator of a given nilpotent matrix of Jordan type $P$ is \emph{stable} - has parts differing pairwise by at least two - and was determined by R. Basili. The second two authors, with B. Van Steirteghem and R. Zhao determined a rectangular table of partitions $\mathfrak D^{-1}(Q)$ having a given stable partition $Q$ as the Jordan type of its maximum nilpotent commutator. They proposed a box conjecture, that would generalize the answer to stable partitions $Q$ having $\ell$ parts: it was proven recently by J.~Irving, T. Ko\v{s}ir and M. Mastnak.\par
    Using this result and also some tropical calculations, the authors here determine equations defining the loci of each partition in $\mathfrak D^{-1}(Q)$, when $Q$ is stable with two parts. The equations for each locus form a  complete intersection. The  authors propose a conjecture generalizing their result to arbitrary stable $Q$.
    
\end{abstract}

\maketitle
\tableofcontents

\section{Introduction}
We let $\kk$ be an infinite field, and let $B$ be an $n\times n$ nilpotent matrix over $\kk$. The \emph{Jordan type} $P_B$ of $B$ is the partition of $n$ determined by the sizes of the blocks in a Jordan block matrix conjugate to $B$.  Although the Jordan decomposition has been known for over a hundred years \cite{Jo,Br} the problem of determining which pairs of partitions $(P,P')$ are possible for two commuting nilpotent matrices $B,B'$ did not receive attention, until in 2008 three research groups - in Ljubljana, Moscow, and Perugia-Boston became interested independently, and wrote about the problem \cite{Ob,Pan,BaI}.\vskip 0.2cm
Since the variety $\mathcal N_B$ parametrizing nilpotent matrices commuting with $B$ is irreducible, there is a maximum Jordan partition $\mathfrak D(P)$ in the dominance orbit closure order of a matrix $A\in \mathcal N_B$. R.~Basili had determined the number of parts $\ell(\mathfrak D(P))$ of $\mathfrak D(P)$, to be the minimum number of almost rectangular partitions (largest and smallest part differ by at most one) needed to cover $P$ \cite{Ba1}. P. Oblak determined the largest part of $\mathfrak D(P)$, and also made a recursive conjecture for determining $\mathfrak D(P)$ \cite{Ob}.  This was studied in several subsequent papers \cite{KO,Pan, Kh1,BaIK,IK}; in particular $\mathfrak D(P)$ was shown by T. Ko\v{s}ir and P. Oblak in \cite{KO} to be stable $\mathfrak D(\mathfrak D(P))=\mathfrak D(P)$; equivalently, $\mathfrak D(P)$ has parts differing pairwise by at least two \cite[Theorem 1]{BaI},\cite{Pan}, known as super-distinct.  The third author determined the smallest part of $\mathfrak D(P)$ \cite{Kh2}; thus, in 2014 $\mathfrak D(P)$ was known, given $P$, when $\mathfrak D(P)$ had up to three parts.  Finally, in 2022, R. Basili proved the Oblak conjecture \cite{Ba2}; a new proof using the Burge code is given in \cite[Theorem 34]{IKM}.  For an account of these developments see the survey by L. Khatami \cite{Kh3}. For a broader survey of problems in commuting nilpotent matrices see \cite{NSi, JeSi}.\par
But what is the set $\mathfrak D^{-1}(Q)$ for $Q$ a stable partition? R. Zhao conjectured in 2014 that for $Q=(u,u-r), r\ge 2$ there should be a rectangular $(r-1)\times (u-r)$ table $\mathcal T(Q)$ of partitions forming the set $\mathfrak D^{-1}(Q)$, and that the element $P_{i,j}(Q)$ in the $i$-row, $j$-column of $\mathcal T(Q)$ should have $(i+j)$ parts.  This was proven by the second two authors with B. Van Steirteghem and R. Zhao in \cite{IKVZ2}. They also proposed a \emph{Box Conjecture} \cite[Conjecture 4.11]{IKVZ2}, generalizing their result to all stable partitions $Q$. The Box Conjecture was recently proven in full generality in a striking paper by J. Irving, T. Ko{\v{s}}ir, and M. Mastnak~\cite{IKM}. We next state their result. 
We label the box as in the original conjecture - as that will be convenient for us.\par\noindent
We write the stable partition $Q=(q_1,q_2,\ldots,q_\ell)$ with parts $q_1>q_2>\cdots >q_\ell$ satisfying $q_i-q_{i+1}\ge 2$ for $1\le i\le \ell-1$ and define a key $S=(s_1,s_2,\ldots,s_\ell)$:
\begin{equation}
\begin{cases}
s_i&=q_i-q_{i+1}-1 \text{ for }1\le i<\ell\\
s_\ell&=q_\ell.
\end{cases}
\end{equation}
\begin{theorem*}\label{thm:box}
    \cite{IKM} Given the stable partition $Q$ with key $S$, there is an $s_1\times s_2\times \cdots \times s_\ell$ box $\mathfrak D^{-1}(Q)$
of partitions $P^Q_I, I=(i_1,\ldots, i_\ell), 1\le i_u\le s_u$ such that $P^Q_I$ has $|I|=\sum i_u$ parts, and is the partition whose Burge code (see Definition \ref{burgecodedef}) is
\begin{equation}\label{Burgegeneraleq}
\alpha^{s_\ell-i_\ell}\beta^{i_\ell}\alpha^{s_{\ell-1}-i_{\ell-1}+1}\beta^{i_{\ell-1}}\alpha^{s_{\ell-2}-i_{\ell-2}+1}\beta^{i_{\ell-2}}\ldots \alpha^{s_1-i_1}\beta^{i_1}\alpha.
    \end{equation}

\end{theorem*}

When $Q = (u,u-r)$ where $u>r\ge 2$ we will denote the entries in this $(r-1)\times (u-r)$ table of partitions by $P^Q_{k,\ello}$ where $1\le k\le r-1$ and $1\le \ello\le u-r$.  The Burge code for $P^Q_{k,\ello}$ is 
\begin{equation}\label{Burgespecialeq}
    \alpha^{u-r-\ello}\beta^{\ello}\alpha^{r-k}\beta^k\alpha.
\end{equation}

\begin{definition}
    For a stable partition $Q$, let $J_Q$ denote a Jordan matrix of Jordan type $Q$ and let $\mathcal N_{J_Q}$ denote the set of nilpotent matrices commuting with $J_Q$, i.e., 
    \[
    \mathcal N_{J_Q} = \{B\in M_{n,n}(\kk): B^n = 0 \text{ and } BJ_Q = J_Q B\}.
    \]
    Let 
    \[
    W^Q_P = \{B \in \mathcal N_{J_Q} : P_B = P\}
    \]
    where $Q = \mathfrak D(P)$ is the generic Jordan type commuting with $P$. Let $\displaystyle\mathcal W_Q = \bigcup_{P:\mathfrak D(P)=Q} W^Q_P$.\par
    When $Q=(u,u-r), r\ge 2$ we will denote by $W^Q_{k,\ello}=W^Q_P, P=P^Q_{k,\ello}$ of Equation~\eqref{Burgespecialeq}.
\end{definition}
\begin{remark}
   When $Q$ is a stable partition, the nilpotent commutator $\mathcal N_{J_Q}$ is a linear space.
   For any $Q$ we have that $W^Q_P$ and $\displaystyle\mathcal W_Q$ are locally closed algebraic subsets of $\mathcal N_{J_Q}$. 
\end{remark}
Our main result is the following theorem giving the equations for the loci $\overline W^Q_P, P\in \mathfrak D^{-1}(Q)$ when $Q$ is stable and has two parts.
\begin{theorem}\label{thm:main}
    When $Q = (u,u-r)$ is a stable partition and $\mathfrak D(P) = Q$, the closure of $W^Q_P$ in $\mathcal N_{J_Q}$ is an irreducible complete intersection of codimension $\ell(P)-2$. The closure of the locus $W^Q_P$ where $ P=P^Q_{k,\ello}$ of Equation \eqref{Burgespecialeq} is $X^Q_{k,\ello}$, defined by the equations $E^Q_{k,\ello}$ of Definition \ref{eqndef}.
\end{theorem}
We develop in the next sections the tools needed. First is the Burge code for a partition $P$, following \cite{IKM}; we interpret it using almost rectangular subpartitions of $P$. The equations $E^Q_{k,\ello}$ we define (Definition \ref{eqndef}) are complete intersections (Proposition \ref{prop:ci}). The Jordan type of a matrix $B\in \mathcal N(Q)$ is determined by the set of ranks of powers of $B$: in Section \ref{tropicalsec} we develop tropical calculations to determine these ranks. We show that the locus $W^Q_{k,\ello}$ is included in the equation locus $X^Q_{k,\ello}$ (Proposition \ref{prop:inclusion}); we then show that the equations we give for the entries in the last column of $\mathcal T(Q)$ determine the partitions in the last column (Proposition \ref{prop:generic}). We complete the proof of Theorem \ref{thm:main} for columns $\ell<(u-r)$ by an induction on $u-r$. It is the Burge code and in particular the box-conjecture result comparing the code for $P$ and for $\partial P$ (\cite[Theorem 23]{IKM}) that allows the key induction argument.
\vskip 0.2cm Our work here leaves open the question of whether the partition in the Table Theorem~1.1 of \cite{IKVZ2} in position $(k,\ello)$ of $\mathcal T(Q)$ is the same as the partition $P^Q_{k,\ello}$ that we give here in terms of its Burge code in Equation \eqref{Burgespecialeq}. In fact, a careful study of the loci of the equations $X^Q_{k,\ello}$ confirms that the equation also determines the partition in the $(k,\ello)$  position for the Table Theorem of \cite{IKVZ2} (Proposition \ref{equivprop} below): due to the length of the verification, and potentially limited interest, we have not included the proof here, but have posted it on arXiv \cite{BIK}. 
\vskip 0.2cm
We give in Section \ref{intersectsec} an example illustrating the question of intersection of the closures of table loci in $\mathfrak D^{-1}(Q).$
In Section \ref{generalsec} we conjecture a generalization of Theorem \ref{thm:main} to stable $Q$ having an arbitrary number of parts.
\section{Notations and background}

A major tool in \cite{IKM} is the \emph{Burge code} of a partition \cite{Bu1,Bu2} This gives a bijection between the set of integer partitions and the set of finite binary words. In order to define this bijection, we first recall the basic definitions given in \cite{IKM} for sequences of finitely supported sequences of nonnegative integers. For a partition $P = (\lambda_1,\lambda_2,\dots,\lambda_s)$ the \emph{frequency sequence} is given by $f_i = \#\{j: \lambda_j=i\}$, for $i\ge 1$, and it is denoted by $\mathcal{F}(P)$. Let $\mathcal{F}$ denote the set of all finitely supported sequences of nonnegative integers. For $f=(f_1,f_2, \dots)\in \mathcal{F}$, the {\it backward pairs} of $f$ are defined as pairs $(f_{j-1},f_j)$ that are obtained by parsing $f$ from right to left and pairing $f_{j-1}, f_j$ whenever $f_j>0$. Suppose that backwards pairs of $f$ are $(f_{j_i-1},f_{j_i})$ for $i=1, \dots m$. The set $\mathsf{R}(f)$ is defined as $\{j_1,j_2,\dots,j_m\}.$ This set is partitioned by the following subsets. 
\[
\begin{array}{cc}
   \mathcal{A}  & = \{f\in \mathcal{F}: 1\not \in \mathsf{R}(f)\},\\
   \mathcal{B}  & = \{f\in \mathcal{F}: 1 \in \mathsf{R}(f)\}. 
\end{array}
\]
Finally, the mapping $\partial: \mathcal{F} \to \mathcal{F}$ is defined as follows. For $f\in \mathcal{F}$, $\partial{f}$ is obtained from $f$ by replacing $(f_{j-1},f_j)$ with $(f_{j-1}+1,f_j-1)$ for each $j\in \mathsf{R}(f)$.

For a partition $P$, we denote $\mathsf{R}(\mathcal{F}(P))$ by $\mathsf{R}(P)$ and $\partial \mathcal{F}(P)$ by $\partial P$. We also say $P$ is in $\mathcal{A}$ (respectively, $\mathcal{B}$)) if $\mathcal{F}(P)\in \mathcal{A}$ (respectively, $\mathcal{B}$). 

Recall that a partition is called {\it almost rectangular} if its biggest and smallest parts differ by at most 1. We use the notation $([m]^k)$ to denote the almost rectangular partition of $m$ into $k$ parts when $m\ge k$. Consider a partition $P$ with frequency sequence $\mathcal{F}(P)=(f_1, f_2,\dots)$. Then $f_j>0$ if and only if $P$ includes at least one part of size $j$. Therefore, backward pairs of $P$ can be obtained by decomposing $P$ into (maximal) almost rectangular subpartitions of $P$ starting from the top (from largest part to smallest). The set $\mathsf{R}(P)$ then consists of the largest part of each almost rectangular subpartition in such a decomposition. Visualizing $P$ by its Ferrers diagram in Figure~\ref{fig:partial} we note that in order to obtain $\partial P$, for each almost rectangular subpartition of $P$ described above, we remove one box from the lowest row corresponding to its largest part.  

\begin{example}
    Let $P=(7,6,5^3,2^2,1)$ with $f=\mathcal{F}(P)=(1,2,0,0,3,1,1)$. Looking at the frequency sequence, the backward pairs are $(1,1)$, $(0,3)$, and  $(1,2)$. Alternatively, looking at the partition, starting form the top decomposing $P$ into almost rectangular subpartitions, we get $(7,6)$, $(5^3)$ and $(2^2,1)$. We have $\mathsf{R}(P)=\{7, 5, 2\}$. We note that in this case $P\in \mathcal{A}$ since $1\notin \mathsf{R}(P)$.

    Figure~\ref{fig:partial} illustrates the process of obtaining $\partial P$, as well as $\partial^2P$. We note the change in the almost rectangular decomposition going from $P$ to $\partial P$. As it is evident from the figure, $P$ and $\partial P$ are both in $\mathcal{A}$, while $\partial^2 P$ is in $\mathcal{B}$ since $1\in \mathsf{R}(\partial^2 P)$. We also note that $\partial$ decreases the length (the number of parts) of the partition by one when the partition is in $\mathcal{B}$, and preserves the length when the partition is in $\mathcal A.$
\end{example}

\begin{definition}\label{burgecodedef}
    The Burge code for a partition $P$ is the binary word $w_1w_2\cdots w_m$ where 
\[
w_i=\begin{cases} 
\alpha,& \text{if $P_i=\partial^{i-1}P\in \mathcal A$},\\
\beta,& \text{if $P_i=\partial^{i-1}P\in \mathcal B$}.
\end{cases}
\]
and $m = \min\{i:\partial^{i-1}P = (0)\}$.
\end{definition} 

\begin{figure}\label{figure:partial}
\centering
\begin{tikzpicture}[scale=0.75]
    \foreach \i in {1,...,6}{
    \draw (\i/2,4)--(\i/2,3);
    \draw (0,3.5)--(3,3.5);
    \draw [line width=0.5mm] (0,3)--(0,4)--(3.5,4)--(3.5,3.5)--(3,3.5)--(3,3)--(0,3);
    \draw [fill=lightgray] (3,3.5) rectangle (3.5,4);
    }    

    \foreach \i in {1,...,5}{
    \draw (\i/2,3)--(\i/2,1.5);
    \draw (0,2.5)--(2.5,2.5);
    \draw (0,2)--(2.5,2);
    \draw [line width=0.5mm] (0,3)--(0,1.5)--(2.5,1.5)--(2.5,3)--(0,3);
    \draw [fill=lightgray] (2,1.5) rectangle (2.5,2);
    }    

  \foreach \i in {1}{
    \draw (\i/2,0.5)--(\i/2,1.5);
    \draw (0,1)--(1,1);
    \draw (0,0.5)--(0.5,0.5);
    \draw [line width=0.5mm] (0,1.5)--(0,0)--(0.5,0)--(0.5,0.5)--(1,0.5)--(1,1.5)--(0,1.5);
    \draw [fill=lightgray] (0.5,0.5) rectangle (1,1);
    }  
    \draw node at (1.5,-1) {$P=(7,6,5^3,2^2,1)$};
    \foreach \i in {1,...,6}{
    \draw (5+\i/2,4)--(5+\i/2,3);
}
\foreach \i in {1,...,4}{
    \draw (5+\i/2,3)--(5+\i/2,2);
}
    \draw (5,3.5)--(5+3,3.5);
    \draw (5,3)--(5+3,3);
    \draw (5,2.5)--(5+2.5,2.5);
    \draw (5,2)--(5+2,2);
    \draw [line width=0.5mm] (5,4)--(5,2)--(7.5,2)--(7.5,2)--(7.5,3)--(8,3)--(8,4)--(5,4);
    \draw [fill=lightgray] (7.5,3) rectangle (8,3.5);
\foreach \i in {1,...,4}{
    \draw (5+\i/2,2)--(5+\i/2,1.5);
}
    \draw [line width=0.5mm] (5,2)--(5,1.5)--(7,1.5)--(7,2);
    \draw [fill=lightgray] (6.5,1.5) rectangle (7,2);

    \draw [line width=0.5mm] (5,1.5)--(5,0)--(5.5,0)--(5.5,1)--(6,1)--(6,1.5);
    \draw (5.5,1.5)--(5.5,1);
    \draw (5,1)--(5.5,1);
    \draw (5,0.5)--(5.5,0.5);
    \draw [fill=lightgray] (5.5,1) rectangle (6,1.5);
    \draw node at (6.5,-1) {$\partial P=(6^2,5^2,4,2,1^2)$};

 \foreach \i in {1,...,5}{
    \draw (10+\i/2,4)--(10+\i/2,3);
}
\foreach \i in {1,...,4}{
    \draw (10+\i/2,3)--(10+\i/2,2);
}
    \draw (10,3.5)--(10+3,3.5);
    \draw (10,3)--(10+2.5,3);
    \draw (10,2.5)--(10+2.5,2.5);
    \draw (10,2)--(10+1.5,2);
    \draw [line width=0.5mm] (10,4)--(10,2)--(7.5+5,2)--(7.5+5,2.5)--(7.5+5,3.5)--(8+5,3.5)--(8+5,4)--(5+5,4);
    \draw [fill=lightgray] (7.5+5,3.5) rectangle (8+5,4);
\foreach \i in {1,...,3}{
    \draw (10+\i/2,2)--(10+\i/2,1.5);
}
    \draw [line width=0.5mm] (5+5,2)--(5+5,1.5)--(6.5+5,1.5)--(6.5+5,2);
   \draw [fill=lightgray] (11,1.5) rectangle (11.5,2);

    \draw [line width=0.5mm] (5+5,1.5)--(5+5,0)--(5.5+5,0)--(5.5+5,1.5);
    \draw (5.5+5,1.5)--(5.5+5,1);
    \draw (5+5,1)--(5.5+5,1);
    \draw (5+5,0.5)--(5.5+5,0.5);
    \draw [fill=lightgray] (10,0) rectangle (10.5,0.5);
    \draw node at (5+6.5,-1) {$\partial^2 P=(6,5^3,3,1^3)$};

 \foreach \i in {1,...,5}{
    \draw (15+\i/2,4)--(15+\i/2,3);
}
\foreach \i in {1,...,4}{
    \draw (15+\i/2,3)--(15+\i/2,2);
}
    \draw (15,3.5)--(15+2.5,3.5);
    \draw (15,3)--(15+2.5,3);
    \draw (15,2.5)--(15+2.5,2.5);
    \draw (15,2)--(15+1.5,2);
    \draw [line width=0.5mm] (15,4)--(15,2)--(7.5+5+5,2)--(7.5+5+5,2.5)--(7.5+5+5,4)--(5+5+5,4);
    \draw [fill=lightgray] (17,2) rectangle (17.5,2.5);
\foreach \i in {1,...,2}{
    \draw (10+5+\i/2,2)--(10+5+\i/2,1.5);
}
    \draw (15,1)--(15.5,1);
    \draw (15,1.5)--(15.5,1.5);
    \draw [line width=0.5mm] (5+5+5,2)--(5+5+5,0.5)--(5.5+5+5,0.5)--(5.5+5+5,1.5)--(6+10,1.5)--(6+10,2);
    \draw [fill=lightgray] (15.5,1.5) rectangle (16,2);
    
    \draw node at (5+6.5+5,-1) {$\partial^3 P=(5^4,2,1^2)$};
\end{tikzpicture}
    \caption{Finding $\partial P$, $\partial^2 P$, and $\partial^3 P$ for $P=(7,6,5^3,2^2,1)$. Each partition is decomposed into disjoint almost rectangular subpartitions, starting from the top. In the figure, these almost rectangular subpartitions are separated from each other by thick borders within each partition. 
    The map $\partial$ acts on a partition by removing one box from the lowest row of the largest part of each almost rectangular subpartition. Those boxes are shaded in gray for each partition.}

    \label{fig:partial}
\end{figure}
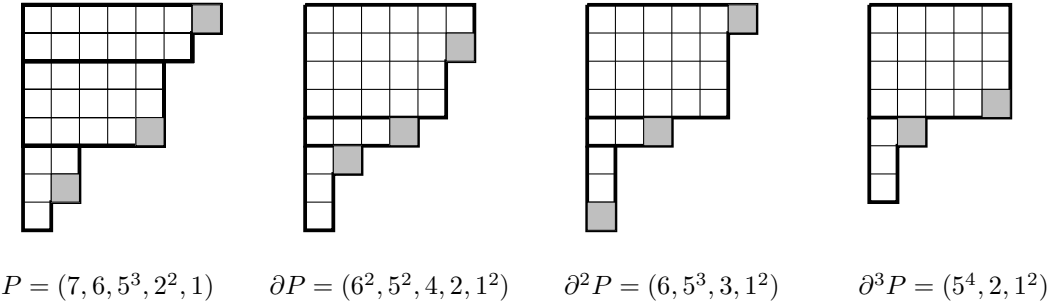

\begin{example}
    Let $P = (5,4,3,3,3,2,2,1) = (5,4)|(3,3,3,2,2)|(1)$. We get 
    \[
    \begin{array}{rcl}
    P_1 = \partial^0 P &=& (5,4)|(3,3,3,2,2)|(1)\\
    P_2 =\,\, \partial P &=&  (4,4,3,3)|(2,2,2)\\
    P_3 = \partial^2 P &=&  (4,3,3,3)|(2,2,1)\\
    P_4 = \partial^3 P &=&  (3,3,3,3,2)|(1,1)\\
    P_5 = \partial^4 P &=&  (3,3,3,2,2)|(1)\\
    P_6 = \partial^5 P &=&  (3,3,2,2,2)\\
    P_7 = \partial^6 P &=&  (3,2,2,2,2)\\
    P_8 = \partial^7 P &=&  (2,2,2,2,2)\\
    P_9 = \partial^8 P &=&  (2,2,2,2,1)\\
    \end{array}
    \begin{array}{rcl}
    P_{10} =\, \partial^9 P &=&  (2,2,2,1,1)\\
    P_{11} = \partial^{10} P &=&  (2,2,1,1,1)\\
    P_{12} = \partial^{11} P &=&  (2,1,1,1,1)\\
    P_{13} = \partial^{12} P &=&  (1,1,1,1,1)\\
    P_{14} = \partial^{13} P &=&  (1,1,1,1)\\
    P_{15} = \partial^{14} P &=&  (1,1,1)\\
    P_{16} = \partial^{15} P &=&  (1,1)\\
    P_{17} = \partial^{16} P &=&  (1)\\
    P_{18} = \partial^{17} P &=&  (0)\\
    \end{array}
    \]
    and the Burge code is $\beta\alpha\alpha\beta\beta\alpha\alpha\alpha\alpha\alpha\alpha\alpha\beta\beta\beta\beta\beta\alpha=\beta\alpha^2\beta^2\alpha^7\beta^5\alpha$. We get $\mathfrak D(P) = (17,5,1)$ since $P_1$, $P_5$ and $P_{17}$ are the last elements in $\mathcal B$ for each of the three consecutive sequences of $P_i$ in $\mathcal B$. We can see more generally that once $\partial^i P = ([m]^k) $ is an almost rectangular partition,  we get that $\partial^{i+j}P = \partial^j([m]^k) = ([m-j]^k)$, for $j=0,1,\dots,m-k$ and $\partial^{m-k}([m]^k)=([k]^k) =(1^k)$. Hence the Burge code for $([m]^k)$ is $\alpha^{m-k}\beta^k\alpha$, which explains the $\alpha^7\beta^5\alpha$ in the example above. 
\end{example}

When $Q = (u,u-r)$ with $u>r\ge 2$, we have that the matrices in $\mathcal N_{J_Q}$ can be written as 
\begin{equation}\label{B-matrix}
B = \begin{bmatrix}
    0&a_1&a_2&a_3&\cdots&a_r&a_{r+1}&\cdots&a_{u-1}&g_0&g_1&g_2&\dots&g_{u-r-1}\\
    0&0&a_1&a_2&\cdots&a_{r-1}&a_{r}&\cdots &a_{u-2}&0&g_0&g_1&\dots&g_{u-r-2}\\
    \vdots&\vdots&\vdots&\vdots&\ddots&\vdots&\vdots&\ddots&\vdots&\vdots&\vdots&\vdots&\ddots&\vdots \\
    0&0&0&0&\cdots&a_{2r-u+1}&a_{2r-u+2}&\cdot&a_{r}&0&0&0&\cdots&g_0\\
    \vdots&\vdots&\vdots&\vdots&\ddots&\vdots&\vdots&\vdots&\vdots&\ddots&\vdots \\
    0&0&0&0&\cdots&0&0&\cdots&a_1&0&0&0&\dots&0\\
    0&0&0&0&\cdots&0&0&\cdots&0&0&0&0&\dots&0\\
    0&0&0&0&\cdots&h_0&h_1&\cdots&h_{u-r-1}&0&b_1&b_2&\dots&b_{u-r-1}\\
    0&0&0&0&\cdots&0&h_0&\cdots&h_{u-r-2}&0&0&b_1&\dots&b_{u-r-2}\\
    \vdots&\vdots&\vdots&\vdots&\ddots&\vdots&\vdots&\ddots&\vdots&\vdots&\vdots&\vdots&\ddots&\vdots \\
    0&0&0&0&\cdots&0&0&\cdot&h_{0}&0&0&0&\cdots&0\\
\end{bmatrix}
\end{equation}
Using this we can get an affine coordinate ring of $\mathcal N_{J_Q}$ as \[\kk[\mathcal N_{J_Q}] = \kk[a_1,a_2,\dots,a_{u-1},b_1,b_2,\dots,b_{u-r-1},g_0,g_1,\dots,g_{u-r-1},h_0,h_1,\dots,h_{u-r-1}].\]

\begin{definition}\label{eqndef}
    For $Q = (u,u-r)$, where $u>r\ge 2$, and $1\le k\le r-1$ and $1\le \ello \le u-r$ let 
    \[
    E^Q_{k,\ello} = \{a_1,a_2,\dots,a_{k-1},b_1,b_2,\dots b_{\ello-1}\} \subseteq \kk[\mathcal N_{J_Q}]
    \]
    when $k+\ello\le r$; and let 
   \[\begin{split}
        E^Q_{k,\ello} &= \{a_1,a_2,\dots,a_{k-1},b_1,b_2,\dots b_{r-k-1},a_kb_{r-k}-g_0h_0, \\&a_kb_{r+1-k}+a_{k+1}b_{r-k}-g_0h_1-g_1h_0,\dots,\sum_{j=0}^{k+\ello-r-1}(a_{k+j}b_{\ello-1-j}-g_jh_{k+\ello-r-1-j})\}\subseteq \kk[\mathcal N_{J_Q}]
   \end{split}
    \]
    if $k+\ello\ge r+1$.    Let $I^Q_{k,\ello} = \left(E^Q_{k,\ello}\right)$ be the ideal generated by $E^Q_{k,\ello}$ in $\kk[\mathcal N_{J_Q}]$ and let $X^Q_{k,\ello} = V(I^Q_{k,\ello})$ be the variety defined by it.
\end{definition}

\begin{remark}\label{2.5rem}
    In the region where $k+l\le r$, there are only linear equations and we have the two inclusions $E^Q_{k-1,l}\subseteq E^Q_{k,l}$ and $E^Q_{k,l-1}\subseteq E^Q_{k,l}$. When $k+l>r$ we also have quadratic equations and we only have the inclusion $E^Q_{k,l-1}\subseteq E^Q_{k,l}$.
\end{remark}

\begin{proposition}\label{prop:ci}
    When $u>r\ge 2$, the variety $X^{(u,u-r)}_{k,\ello}$ is a reduced irreducible complete intersection of codimension $k+\ello-2$, for $1\le k\le r-1$ and $1\le \ello\le u-r$.
\end{proposition}

\begin{proof}
    For the linear equations, the statement is clear. When focusing on the quadratic equations
    \[
    \begin{array}{rcl}
    a_kb_{r-k} &= &g_0h_0 \\
    a_kb_{r+1-k}+a_{k+1}b_{r-k} & = & g_0h_1+g_1h_0\\
    &\vdots\\
    a_kb_{\ello-1}+a_{k+1}b_{\ello-2}+\cdots+a_{2k+\ello-1}b_{r-k} & = & g_0h_{k+\ello-r}+g_1h_{k+\ello-r-1}+\cdots +g_{k+\ello-r}h_0
    \end{array}
    \]
    we can successively solve for $b_{r-k},b_{r-k+1},\dots,b_{\ello-1}$ as long as $a_k\ne 0$ and we can solve for $a_k,a_{k+1},\dots,a_{2k-r+\ello-1}$ as long as $b_{r-k}\ne 0$.  If both are zero, we get a locus of higher codimension. Hence an open dense subset of $X^{(u,u-r)}_{k,\ello}$ can be parametrized by an affine space of dimension $4u-3r-2-(k+\ello-2) = 4u-3r-k-\ello$. This gives us the irreducibility and codimension. Since each variable that we solve for occurs only linearly in the equations, the ideal is radical and since the number of equations equals the codimension it is a complete intersection.
\end{proof}

\begin{remark}
    For $n\in \mathbb N$ we can consider the $2\times 2$-matrices with entries in $\kk[t]/(t^n)$ with vanishing determinant as the $(n-1)$-jet bundle of the variety of rank one matrices in $M_{2,2}(\kk)$. This has the same quadratic equations as $X^Q_{k,\ello}$ after a shift of indices and with $n=k+\ello-r$. We can see this as it parametrizes maps 
    \[
    \operatorname{Spec}\left(\kk[t]/(t^n)\right)\longrightarrow 
    \operatorname{Spec}\left(\kk[a,b,g,h]/(ab-gh)\right)
    \]
    where the right hand side is the affine cone over a smooth quadric surface in $\mathbb P^3_\kk$. 
\vskip 0.2cm\par In discussions with the second author, T. Ko\v{sir}, P. Oblak, and K. \v{S}ivic pointed out that the equations obtained had similarities with the work of \cite{KS,KS2}. One difference is that we consider several orders of vanishing for different elements of our matrix $M_B$ of Equation~\eqref{MBeqn}.\par
We thank the referee for pointing out that the equations for loci $X^Q_{k,\ello}$ we show here have structure with some similarity to those shown by M. Neubauer and B. A. Sethuraman for commuting pairs in the centralizer of a $2$-regular matrix (see \cite{NeS}). The difference in the
linear equations is that the equations for the varieties $X^Q_{k,\ello}$ are given by the parameters close to the
diagonal of the larger diagonal Toeplitz block, while in the case of \cite{NeS} they come from the parameters
in the upper-right corner of the Toeplitz block. The quadratic equations in both cases are equations
of jets over a determinantal variety. The difference is that, in the case of $X^Q_{k,\ello}$ these are jets over
rank one matrices of size $2 \times 2$, while in the case of \cite{NeS} these are jets over rank one matrices that are $3\times 2$. T. Ko\v{sir} and B. A. Sethuramam then studied jets over general determinantal varieties in \cite{KS,KS2}.
Further discussion of these jets are found in \cite{Do,Ma,Yu}, while computational software is given in \cite[\S 4]{GI}.
\end{remark}
For a matrix $B$ in $\mathcal N_{J_Q}$ we can associate to $B$, as in (\ref{B-matrix}), a matrix 
\begin{equation}\label{MBeqn}
M_B = \begin{bmatrix}
    a&g\\
    ht^r&b
\end{bmatrix} = \begin{bmatrix}
    a_1t+a_2t^2+\cdots+a_{u-1}t^{u-1}&g_0+g_1t+\cdots+g_{u-r-1}t^{u-r-1}\\
    h_0t^r+a_1t^{r+1}+\cdots+h_{u-r-1}t^{u-1}&g_1t+b_2t^2+\cdots+b_{u-r-1}t^{u-r-1}\\
\end{bmatrix}
\end{equation}
where $a,ht^r\in \kk[t]/(t^u)$ and $g,b\in \kk[t]/(t^{u-r})$. The matrices $B$ and $M_B$ represent the same $\kk$-linear map under the $\kk$-linear isomorphism $\kk^{u}\oplus\kk^{u-r}\longrightarrow \kk[t]/(t^u)\oplus\kk[t]/(t^{u-r})$. In particular, they have the same rank and corank and multiplication by $J_Q$ on $B$ corresponds to multiplication by $t$ on $M_B$.
The matrix $M_B$ also represents a homomorphism in \[
\Hom_{\kk[t]}(\kk[t]/(t^u)\oplus \kk[t]/(t^{u-r}),\kk[t]/(t^u)\oplus \kk[t]/(t^{u-r}))\] and the map 
\[
    \mathcal N_{J_Q} \longrightarrow \mathcal R_Q:=\Hom_{\kk[t]}(\kk[t]/(t^u)\oplus \kk[t]/(t^{u-r}),\kk[t]/(t^u)\oplus \kk[t]/(t^{u-r}))
\]
given by $B\mapsto M_B$ is an injective ring homomorphism.
We will use the power series ring $\kt$ for calculations and then restrict to the quotients by $(t^u)$ and $(t^{u-r})$.

In the endomorphism ring $\Hom_{\kt}(\kt^2,\kt^2)$, there is a subring given by 
\begin{equation}\label{eq:Hr}
    \mathcal H_r = \left\{\begin{bmatrix}
    a&g\\
    ht^r&b
\end{bmatrix}\colon a,b,g,h\in \kt\right\}
\end{equation}
and the natural map $\mathcal H_r \longrightarrow\mathcal R_Q$ is a surjective ring homomorphism. 

\section{Tropical calculations}\label{tropicalsec}

We will use the order of the elements in the matrices in the calculations of the coranks. Recall that for $a\in \kt$, the order of $a$ is the lowest power of $t$ in $a$ that has a nonzero coefficient. For a matrix $B$ we will have 
\[
T(M_B)  = \begin{bmatrix}
    \ord(a)&\ord(g)\\
    r+\ord(h)&\ord(b)
\end{bmatrix}
\]
Since $B$ is nilpotent, $\ord(a)\ge 1$ and $\ord(b)\ge 1$.

\begin{lemma}\label{lemma:det}
    If $M_B = \begin{bmatrix}
    a&g\\
    ht^r&b
\end{bmatrix}$, $a\ne 0$ and 
    $\ord(a)\le r+\min\{\ord (g),\ord(h)\}$, then \[
    \corank B=\min \{\ord(\det M_B),\ord(a)+u-r\}.
    \]
\end{lemma}

\begin{proof}
We have that $\corank B = \dim_\kk\ker B = \dim_\kk \ker M_B$, since $B$ and $M_B$ correspond to the same $\kk$-linear map under the isomorphism $\kk^{u}\oplus\kk^{u-r}\cong\kk[t]/(t^u)\oplus\kk[t]/(t^{u-r})$. If $k=\ord(a)$, we can write 
    \[
    B = \begin{bmatrix}
        a't^k&g\\ht^{r}&b
    \end{bmatrix}
    \]
    where $a'$ is invertible in $\kt$. Row reduction gives 
    \[
    B' = \begin{bmatrix}
        a't^k&g\\0&b-(a')^{-1}ght^{r-k}
    \end{bmatrix}
    \]
    and $\corank B = \corank B'$. The condition $\ord(a)\le \ord(h)+r$ makes the row reduction possible, while the condition $\ord(a)\le \ord(g)+r$ guarantees that the upper right corner of $B'$ does not affect the computation of $\corank B'$ which is given by $k+u-r$ if $b-(a')^{-1}ght^{r-k} = 0$ and by 
    \[
    k + \ord(b-(a')^{-1}ght^{r-k}) = k + \ord(a'b-ght^{r-k}) = k+ \ord((ab-ght^r)t^{-k}) = \ord(\det M_B)
    \]
    if $b-(a')^{-1}ght^{r-k} \neq 0$. We note that we are not assuming above that either $g$, $h$ or $b$ is invertible.
\end{proof}

\begin{remark}\label{rmk:equations}
    The equations $E^{(u,u-r)}_{k,\ello}$ are given by $\ord(a)\ge k$ and $\ord\det(M_B)\ge k+\ello$.   
\end{remark}

We next show that the locus of matrices in $\mathcal N_{J_Q}$ with partition $P^Q_{k,\ello}$ of Equation \eqref{Burgespecialeq} lies inside the locus $X^Q_{k,\ello}$ defined by the equations $E^Q_{k,\ello}$.
\begin{proposition}\label{prop:inclusion}
    For a stable partition $Q = (u,u-r)$ and for $1\le k\le r-1$ and $1\le \ello \le u-r$ we have an inclusion $W^Q_{k,\ello}\subseteq X^Q_{k,\ello}$.
\end{proposition}

\begin{proof}
    We will use induction on $u-r$. The base case is when $u-r=1$ and we have that $E^{(r+1,1)}_{k,1} = \{a_1,a_2,\dots,a_{k-1}\}$ for $k=1,2,\dots,r-1$. For $B\in \mathcal N_{(r+1,1)}$ we have  \[
    M_B = 
    \begin{bmatrix}
        a_1t+a_2t^2+\cdots+a_rt^r&g_0\\
        h_0t^r&0\\
    \end{bmatrix}
    \]
    and the partition $P^Q_{k,1}$ has $k+1$ parts, which means that the  $\corank B = \dim_\kk \ker(M_B)=k+1$. This forces $a_1=a_2=\cdots=a_{k-1}=0$ which shows the inclusion $W^{(r+1,1)}_{k,1}\subseteq X^{(r+1,1)}_{k,1}$, for $k=1,2,\dots,r-1$.

    Now assume by induction that we have the inclusion $W^{(u,u-r)}_{k,\ello}\subseteq X^{(u,u-r)}_{k,\ello}$ for $u-r = m\ge 1$. The restriction of the equations $E^{(u+1,u+1-r)}_{k,\ello}$ from $\kk[\mathcal N_{J_{(u+1,u+1-r)}}]$ to $\kk[\mathcal N_{J_{(u,u-r)}}]$ gives the equations $E^{(u,u-r)}_{k,\ello}$ when $\ello\le u-r$. By \cite[Theorem 23]{IKM} we have that for $B\in W^{(u+1-r)}_{k,\ello}$ the Jordan type of $B|_{\im(J_{(u+1,u+1-r)}
)}$ is $\partial (P^{(u+1,u+1-r)}_{k,\ello}) = P^{(u,u-r)}_{k,\ello}$ which is in $\mathcal N_{J_{\partial (u+1,u+1-r)}} = \mathcal N_{J_{(u,u-r)}}$ and $\mathfrak D(\partial P^{(u+1,u+1-r)}_{k,\ello}) = (u,u-r)$. Hence we conclude by induction on $u-r$ that $W^{(u+1,u+1-r)}_{k,\ello} \subseteq X^{(u+1,u+1-r)}_{k,\ello}$ when $\ello\le u-r$.

    When $\ello=u+1-r$ we get that $\partial P^{(u+1,u+1-r)}_{k,\ello} = P^{(u,u-r)}_{k,\ello-1} = P^{(u,u-r)}_{k,u-r} $. By the same argument as above we get that $W^{(u+1,u+1-r)}_{k,u+1-r} \subseteq X^{(u+1,u+1-r)}_{k,u-r}$. The number of parts of $P^{(u+1,u+1-r)}_{k,u+1-r}$ is $k+u+1-r$ which means that $\corank B = k+u+1-r$. By Lemma~\ref{lemma:det} and Remark~\ref{rmk:equations} this gives the last equation in $E^{(u+1,u+1-r)}_{k,u+1-r}$ when $k+u+1-r> r$ since it is an irreducible quadric. Hence $W^{(u+1,u+1-r)}_{k,u+1-r} \subseteq X^{(u+1,u+1-r)}_{k,u+1-r}$ for $k+u+1-r>r$. When $k+u+1-r\le r$, the equation given by $\ord\det(M_B)\ge k+u+1-r$ is $a_kb_{u+1-r}=0$. If $a_k=0$ we get that $B|_{\im J_Q}$ is in $X^{(u,u-r)}_{k+1,u-r}$, but this would imply that it has at least $k+1+u-r$ parts. But $P^{(u,u-r)}_{k,u-r}$ has $k+u-r$ parts. Hence we conclude that $b_{u+1-r}=0$ and $W^{(u+1,u+1-r)}_{k,u+1-r} \subseteq X^{(u+1,u+1-r)}_{k,u+1-r}$.
\end{proof}

It will be useful to use the following \emph{tropical} notation when computing the order matrices of powers of matrices.

\begin{notation}\label{tropical}
    The notation $a\oplus b = \min\{a,b\}$ will be used for $a,b\in \mathbb N$.
\end{notation}

\begin{lemma}\label{tropicalpowerlemma}
  If $M\in \mathcal H_r$ with $T(M) =
  \begin{bmatrix}
    k&0\\r&\ello
  \end{bmatrix}
$  and if there are no cancellations of leading terms in $M^s$, we have that  $T(M^s)$ is given by 
\[
\textstyle
\begin{bmatrix}
  sk\oplus\frac s2r\oplus((s-2)\ello+r) &
  (s-1)k\oplus(k+\frac{s-2}2r)\oplus (\ello +
  \frac{s-2}2r)\oplus(s-1)\ello\\
 ((s-1)k+r)\oplus(k+\frac{s}2r)\oplus (\ello +
\frac{s}2r)\oplus((s-1)\ello+r)& ((s-2)k+r)\oplus\frac s2r\oplus s\ello
\end{bmatrix}
\]
for even $s\ge 2$ and 
\[
\textstyle
\begin{bmatrix}
  sk\oplus (k+\frac {s-1}2r)\oplus (\ello+\frac {s-1}2r)\oplus((s-2)\ello+r) &
  (s-1)k\oplus\frac{s-1}2r\oplus (s-1)\ello\\
  ((s-1)k+r)\oplus\frac{s+1}2r\oplus ((s-1)\ello+r)&
  ((s-2)k+r)\oplus (k+\frac {s-1}2r)\oplus (\ello+\frac {s-1}2r)\oplus s\ello 
\end{bmatrix}
\]
for odd $s\ge 3$.  
\end{lemma}

\begin{proof}
We start by computing $T(M^s)$ for $s=2$ and $s=3$.
\[
T(M^2) = \textstyle
\begin{bmatrix}
  2k \oplus r & k\oplus \ello\\
(k+r)\oplus (\ello +r)& r\oplus 2\ello
\end{bmatrix}
\]
and 
\[
T(M^3) = \textstyle
\begin{bmatrix}
  3k \oplus (k+r)\oplus(\ello+r) & 2k\oplus r\oplus 2\ello\\
(2k+r)\oplus 2r\oplus (2\ello +r)& (k+r)\oplus (\ello+r)\oplus 3\ello
\end{bmatrix}
\]
which agrees with the formulas.  We then proceed by induction starting with even $s$  and multiply by $M$. The entries of $T(M^{s+1})$ will be
\[\begin{split}
a_{1,1} & \textstyle=  (s+1)k\oplus(k+\frac s2r)\oplus(k+(s-2)\ello+r)  \oplus 
 ((s-1)k+r)\oplus(k+\frac{s}2r)\\&\textstyle\oplus (\ello +
\frac{s}2r)\oplus((s-1)\ello+r) = (s+1)k\oplus(k+\frac s2r)\oplus (\ello +
\frac{s}2r)\oplus((s-1)\ello+r)\\
a_{2,1} &\textstyle = (sk+r)\oplus(\frac s2r+r)\oplus((s-2)\ello+r+r)\oplus((s-1)k+\ello+r)\oplus(k+\ello+\frac s2r)\\&\textstyle \oplus(2\ello+\frac s2r)\oplus(s\ello+r) =  (sk+r)\oplus \frac{s+2}2r \oplus(s\ello+r)\\
a_{1,2} &\textstyle = sk\oplus (2k+\frac{s-2}2r)\oplus (k+\ello+\frac{s-2}2r)\oplus (k+(s-1)\ello)\oplus ((s-2)k+r)\oplus\frac s2r\oplus s\ello \\&\textstyle=
sk\oplus\frac{s}2r\oplus s\ello \\
a_{2,2} & \textstyle= ((s-1)k+r)\oplus(k+\frac{s}2r)\oplus(\ello+\frac{s}2r)\oplus((s-1)\ello+r)\oplus((s-2)k+\ello+r)\oplus(\ello+\frac{s}2r)\\&\textstyle
\oplus((s+1)\ello) =  ((s-1)k+r)\oplus(k+\frac{s}2r)\oplus(\ello+\frac{s}2r)\oplus((s-1)\ello+r)
\end{split}
\]
where we have used that we only need to consider the extreme cases for the three quantities $k$, $r/2$ and $\ello$ since all terms are convex combinations of them with weights summing to $s+1$, $s+2$, $s$ and $s+1$, respectively.

We now multiply from odd powers to even powers to get the entries of $T(M^{s+1})$ as 
\[
\begin{split}
    a_{1,1} &\textstyle= (s+1)k \oplus (2k+\frac{s-1}2r)\oplus(k+\ello+\frac{s-1}{2}r)\oplus(k+(s-2)\ello+r) \oplus((s-1)k+r) \\& \textstyle
    \oplus\frac{s+1}{2}r\oplus((s-1)\ello+r)= (s+1)k\oplus \frac{s+1}2r\oplus((s-1)\ello+r)\\
    a_{2,1} &\textstyle= (sk+r)\oplus(k+\frac{s+1}{2}r)\oplus(\ello+\frac{s+1}{2}r)\oplus((s-2)\ello+2r)\oplus((s-1)k+\ello+r) \\
    &\textstyle\oplus(\ello+\frac{s+1}{2}r)\oplus(s\ello+r)=
    (sk+r)\oplus(k+\frac{s+1}{2}r)\oplus(\ello+\frac{s+1}{2}r)\oplus(s\ello+r) \\
    a_{1,2} &\textstyle= sk\oplus(k+\frac{s-1}{2}r)\oplus(k+(s-1)\ello)\oplus((s-2)k+r)\oplus(k+\frac{s-1}{2}r)\oplus(\ello+\frac{s-1}{2}r)\oplus s\ello
    \\ &\textstyle=
    sk \oplus( k+ \frac{s-1}{2}r)\oplus(\ello +\frac{s-1}{2}r)\oplus s\ello
    \\
    a_{2,2} &\textstyle= ((s-1)k+r)\oplus(\frac{s+1}{2}r)\oplus((s-1)\ello+r)\oplus ((s-2)k+\ello+r)\oplus(\ello+\frac{s}{2}r)\oplus(s+1)\ello\\
    &\textstyle=
    ((s-1)k+r)\oplus\frac{s+1}{2}r\oplus(s+1)\ello
\end{split}
\]
\end{proof}

\begin{lemma}\label{tropicalcoranklemma}
    For $B$ general in $X^{(u,u-r)}_{k,\ello}$, where $1\le k\le r-1$ and $1\le \ello\le u-r$, we have that 
    \[
    \corank B^s = (k+\ello)s \oplus (T(M_B^s)_{1,1}+u-r) \oplus (2u-r),\quad  \text{for $s\ge 0$.}
    \]
\end{lemma}

\begin{proof}
    Observe that since $B$ is general in $X^{(u,u-r)}_{k,\ello}$, the order of $\det(M_B)$ is $k+\ello$.    We have that 
    \[
    T(M_B) = \begin{bmatrix}
        k&0\\r&\ello'
    \end{bmatrix}
    \]
    where $\ello' = \ello\oplus(r-k)$. Hence we have that $\ello'<r$.  Lemma~\ref{tropicalpowerlemma} gives that 
    \[
    T(M^s_B)_{1,1}\le T(M^s_B)_{2,1} = T(M^s_B)_{1,2}+r
    \]
    since $k<r$ and $\ello'<r$ . Hence the conditions in Lemma~\ref{lemma:det} are satisfied for $M_B^s$.
    
    As long as the computation of the determinant in $\kt$ is not affected by the quotients by $(t^u)$ and $(t^{u-r})$, we have that $\det(B^s) = (\det B)^s$, which gives the first term.   
    
    When $(k+\ello)s>T(M_B^s)_{1,1}+u-r$, the computation of the determinant is not valid since we will have a zero row after row reduction when computing modulo $(t^{u-k})$ in the second row. Now, $\corank B^s =T(M_B^s)+u-r$ as long as $M_B^s\ne 0$. When $M_B=0$ we get $\corank B_s=u+(u-r)=2u-r$.
\end{proof}

We next show that the equations we give for the entries in the last column of $\mathcal T(Q)$ determine $P^Q_{k,u-r}.$
 \begin{proposition}\label{prop:generic}
   Let $Q=(u,u-r)$ and assume that $B$ is a generic element of  $X^Q_{k,u-r}$, for some $1\leq k\leq r-1$. Then $P_B=P^Q_{k,u-r}$. 
\end{proposition}
    
\begin{proof}
    Assume that $B$ is a general enough element of $X^Q_{k,u-r}$. We will show that the Burge code of $P_B$ is given by
    \begin{equation}\label{desiredpart}
        \overbrace{\beta \dots \beta}^{u-r}\underbrace{\alpha \dots \alpha}_{r-k}\overbrace{\beta \dots \beta}^{k}\alpha.
    \end{equation}

Since $B$ is general in $X^Q_{k,u-r}$ we have that $\corank B=\operatorname{ord}(\operatorname{det} M_B)=k+u-r$ and  
\[
T\left(M_B\right)=  \begin{bmatrix}k&&0\\r&&\ello'\end{bmatrix},
\]
where $\ello'=(u-r)\oplus (r-k)$.

By Lemma \ref{tropicalcoranklemma}, for $s\geq 1$, we have that
\[
\corank B^s=s(k+\ello)\oplus ((T(M^s_B))_{11}+u-r)\oplus(u+u-r).
\]
and by Lemma \ref{tropicalpowerlemma}, for $s\geq 2$, we have that
\[
(T(M^s_B))_{11}=
	\begin{cases}
	sk\oplus\frac{s}{2}r\oplus((s-2)\ello'+r)\oplus u,&\text{ if $s$ is even},\\ 
	sk\oplus(k+\frac{s-1}{2}r)\oplus(\ello'+\frac{s-1}{2}r)\oplus((s-2)\ello'+r)\oplus u,&\text{ if $s$ is odd}.
	\end{cases}
\]
There are four cases depending on if $2k\le r$ or $2k>r$ and if $\ello'=r-k$ or $\ello'=u-r$.

First assume that $2k\leq r$ and that $\ello'=r-k$. 
Then we have $\ello'=r-k\ge 2k-k = k$ and hence
\[
sk\le \frac{s}{2}r, \quad sk\le (s-2)k+r \le (s-2)\ello'+r.
\]
and 
\[
sk=k+(s-1)k\leq k+\frac{s-1}{2}r\leq \ello'+\frac{s-1}{2}r.
\]
We conclude that $(T(M^s_B))_{11}=sk\oplus u$ for $s\ge 2$.

Now assume that $2k\le r$ and that $\ello'=u-r$. Then for $s=2$ we have $\frac{s}{2}r = (s-2)\ello'+r = r$ and for $s\ge 3$ we have 
\[
(s-2)\ello'\ge u \quad \text{and}\quad \ello'+\frac{s-1}{2}\ge u
\]
We conclude that also in this case we have $(T(M^s_B))_{11}=sk\oplus u$ for $s\ge 2$.

Since $\corank B=k+u-r$, we can write 
\[
\corank B^s=(sk+u-r)\oplus(u+u-r)
=\begin{cases}
     sk+u-r,&\text{for  $1\leq s \leq \frac{u}{k}$},\\ 
     u+u-r, &\text{for  $\frac{u}{k}<s$}
\end{cases}
\]
for $s\geq 1$, when $2k\le r$.

We recall that, considering the Ferrers diagram of $P_B$, the corank of $B^s$ is the total number of boxes in the first $s$ columns of $P_B$. Thus, the calculations above show that $P_B=([u]^k,1^{u-r})$. 

Thus we have that $\partial^iP_B = ([u-i]^k,1^{u-r-i})$, for $i=0,1,\dots,u-r$, which shows that the Burge code starts with $(\beta)^{u-r}$ and $\partial^{u-r}P_B =([r]^k)$. The Burge code for $([r]^k)$ is $\alpha^{r-k}\beta^k\alpha$ since we have that  $\partial^i([r]^k) = \partial([r-i]^k)$, for $i=0,1,\dots,r-k$ and $\partial^{r-k}([r]^k)=([k]^k)=(1^k)$. Hence the Burge code is $\beta^{u-r}\alpha^{r-k}\beta^k\alpha$. See Figure~\ref{fig:Case1} for an illustration of the process.

We now assume that  $2k>r$. Since $k+\ello'\leq r$, we get $\ello'<\frac{r}{2}<k$. Thus for $s\geq 2$, we have 
\[
(s-2)\ello'+r\leq \frac{s-2}{2}r+r=\frac{s}{2}r<sk.
\]
and for  $s\geq 3$ we have 
\[
(s-2)\ello'+r=\ello'+(s-3)\ello'+r\leq \ello'+\frac{s-1}{2}r< k+\frac{s-1}{2}r.
\]
Thus we conclude that 
\(
(T(M^s_B))_{11}=((s-2)\ello'+r)\oplus u.
\) for $s\ge 2$.

Now we consider the two possible cases for $\ello'$. First assume that $\ello' = r-k\le u-r$. Then we get
\[
\corank B^s=\begin{cases}
    k+u-r, &\quad \text{for $s=1$,} \\
     (s-2)(r-k)+u,&\quad \text{for $2\leq s \leq \frac{u-r}{r-k}+2$,}\\
     u+u-r, &\quad\text{for $s>\frac{u-r}{r-k}+2.$}
\end{cases}
\]
Thus, in this case $P_B=([u-r+2(r-k)]^{r-k},1^{u-2(r-k)})$. 
We have that $\partial^iP_B =([u-r+2(r-k)-i]^{r-k},1^{u-2(r-k)-i})$ for $i = 0,1,\dots,u-r$ and $\partial^{u-r}P_B = ([2(r-k)]^{r-k},1^{2k-r})=([r]^k)$. Thus the Burge code for $P_B$ starts with $\beta^{u-r}$. We have already seen that the Burge code for $[r]^k$ is $\alpha^{r-k}\beta^k\alpha$ which shows that the Burge code for $P_B$ is $\beta^{u-r}\alpha^{r-k}\beta^k\alpha$. See Figure~\ref{fig:Case2.1} for an illustration of the process.

Finally assume that $\ello'=u-r \leq r-k$. In this case, 
\[
\corank B^s=\begin{cases}
    k+u-r, & \quad \text{for $s=1$,} \\
    u,& \quad\text{for $s=2$,}\\ 
     u+u-r, & \quad \text{for $s\geq 3$.}
\end{cases}
\]
Therefore, $P_B=(3^{u-r}, 2^{(r-k)-(u-r)}, 1^{u-r+(2k-r)})=([u+r-2k]^{r-k}, 1^{u-2r+2k})$.  We have that $\partial^i P_B =([u+r-2k-i]^{r-k}, 1^{u-2r+2k-i})$ for $i=0,1,\dots,u-r$ and $\partial^{u-r}P_B = ([2(r-k)]^{r-k},1^{2k-r})=(2^{r-k},1^{2k-r})=([r]^k)$, which shows that the Burge code of $P_B$ starts with $\beta^{u-r}$ and continues with the Burge code for $([r]^k)=\alpha^{r-k}\beta^k\alpha$. Hence the Burge code is $\beta^{u-r}\alpha^{r-k}\beta^k\alpha$ as desired. See Figure~\ref{fig:Case2.2} for an illustration of the process.
\end{proof}
\begin{figure}
    \centering
    \begin{tikzpicture}[scale=0.75]
\foreach \i in {0} {
    
    \draw [draw=black, fill=lightgray] (\i+0,0) rectangle (\i+0.5,3);
   
    \draw [->] (\i-0.25,1.5)--(\i-0.25,0.05);
    \draw node at(\i-0.25,2.5) [rotate=90] {\tiny{$u-r+k$}} ;
    \draw [->] (\i-0.25,3.5)--(\i-0.25,4.95);
    \draw (\i+0,3)--(\i+1,3)--(\i+1,3.5)--(\i+1.5,3.5)--(\i+1.5,5)--(\i+0,5)--(\i+0,3);
    \draw [fill=lightgray] (\i+1,3)--(\i+1,3.5)--(\i+1.5,3.5)--(\i+1.5,5)--(\i+2.5,5)--(\i+2.5,4.5)--(\i+2,4.5)--(\i+2,3)--(\i+1,3);
    \draw [->] (\i+2.75,3.75)--(\i+2.75,3.05);
    \draw node at (\i+2.75,4) [rotate=90]{\tiny{$k$}} ;
    \draw [->] (\i+2.75,4.25)--(\i+2.75,4.95);
    
    \draw node at (\i+1,-.5) {\tiny{$P=([u]^k,1^{u-r})$}};
    \draw node at (\i+1,-1.5) {\tiny{Burge code:}};
    
    \draw node at (\i+3.5,-1.75) {\tiny{$\underbrace{\beta\dots \dots \beta}_{u-r}$}};
}

\foreach \i in {5}{
    \draw (\i+0,3)--(\i+1,3)--(\i+1,3.5)--(\i+1.5,3.5)--(\i+1.5,5)--(\i+0,5)--(\i+0,3);
    \draw [fill=lightgray] (\i+0.5,3)--(\i+1,3)--(\i+1,3.5)--(\i+1.5,3.5)--(\i+1.5,5)--(\i+0.5,5)--(\i+0.5,3);
    \draw node at (\i+1,-.5) {\tiny{$\partial^{u-r}P=([r]^k)$}};
    
    \draw node at (\i+2.5,-1.75) {\tiny{$\underbrace{\alpha \dots \alpha}_{r-k}$}};
   
}
\foreach \i in {8}{
    \draw [fill=lightgray] (\i+0.75,3) rectangle (\i+1.25,5);
    \draw node at (\i+1,-.5) {\tiny{$\partial^{u-k}P=(1^k)$}};
    
    \draw node at (\i+2.5,-1.75) {\tiny{$\underbrace{\beta \dots \beta}_{k}$}};
  
}
\foreach \i in {11}{
    \draw node at (\i+1,-.5) {\tiny{$\partial^{u}P=\varepsilon$}};
    
    \draw node at (\i+2,-1.5) {\tiny{$\alpha$}};
    
}
\draw (-1,-2.5) rectangle (14,5.5);
\end{tikzpicture}
    \caption{Illustration for computation of the Burge code in Proposition~\ref{prop:generic} when $2k\leq r$.}
    \label{fig:Case1}
\end{figure}
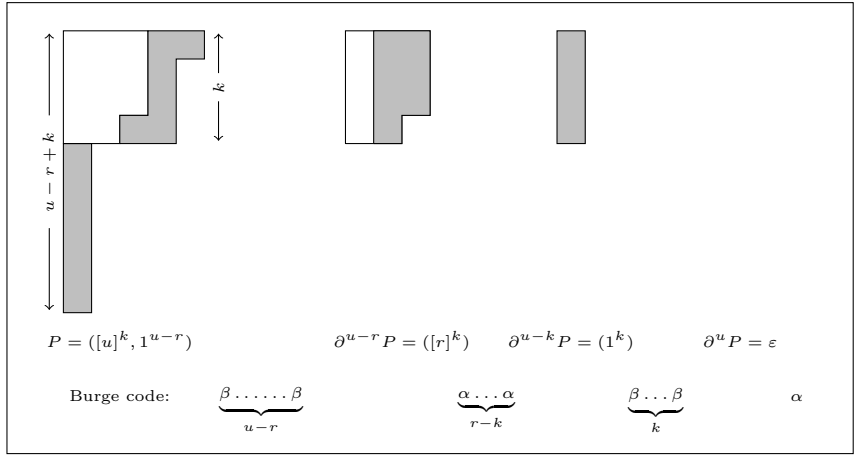

\begin{figure}
    \centering
    \begin{tikzpicture}[scale=0.75]
\foreach \i in {0} {
    
    \draw  (\i+0,0) rectangle (\i+0.5,3.5);
    \draw [draw=black, fill=lightgray] (\i+0,0) rectangle (\i+0.5,2);
    
    \draw [->] (\i-0.25,1.5)--(\i-0.25,0.05);
    \draw node at(\i-0.25,2.5) [rotate=90] {\tiny{$u-r+k$}} ;
    \draw [->] (\i-0.25,3.5)--(\i-0.25,4.95);
    \draw [->] (\i+0.75,0.5)--(\i+0.75,0.05);
    \draw node at(\i+0.75,1) [rotate=90] {\tiny{$u-r$}} ;
    \draw [->] (\i+0.75,1.5)--(\i+0.75,1.95);
    \draw (\i+0,3.5)--(\i+1.5,3.5)--(\i+1.5,4.5)--(\i+2,4.5)--(\i+2,5)--(\i+0,5)--(\i+0,3.5);
    \draw [fill=lightgray] (\i+1,3.5)--(\i+1.5,3.5)--(\i+1.5,4.5)--(\i+2,4.5)--(\i+2,5)--(\i+1,5)--(\i+1,3.5);
    \draw [->] (\i+2.25,3.75)--(\i+2.25,3.55);
    \draw node at (\i+2.25,4.25) [rotate=90]{\tiny{$r-k$}} ;
    \draw [->] (\i+2.25,4.75)--(\i+2.25,4.95);
     \draw node at (\i+0,-.5) {\tiny{$P=([u-r+2(r-k)]^{r-k},1^{u-2(r-k)})$}};
    \draw node at (\i+1,-1.5) {\tiny{Burge code:}};
    
    \draw node at (\i+3.5,-1.75) {\tiny{$\underbrace{\beta\dots \dots \beta}_{u-r}$}};
}

\foreach \i in {5}{
    \draw  (\i+0,2) rectangle (\i+0.5,5);
    
    \draw [fill=lightgray] (\i+0.5,3.5) rectangle (\i+1,5);
    \draw node at (\i+0.75,-.5) {\tiny{$\partial^{u-r}P=(2^{r-k},1^{2k-r})$}};
    
    \draw node at (\i+2.75,-1.75) {\tiny{$\underbrace{\alpha \dots \alpha}_{r-k}$}};
    }
\foreach \i in {8}{
    
    \draw [fill=lightgray] (\i+0.75,2) rectangle (\i+1.25,5);

    \draw node at (\i+1.25,-.5) {\tiny{$\partial^{u-k}P=(1^k)$}};
    
    \draw node at (\i+2.75,-1.75) {\tiny{$\underbrace{\beta \dots \beta}_{k}$}};

}
\foreach \i in {11}{

    \draw node at (\i+1,-.5) {\tiny{$\partial^{u}P=\varepsilon$}};
    
    \draw node at (\i+2,-1.5) {\tiny{$\alpha$}};

}
\draw (-3.75,-2.5) rectangle (14,5.5);
\end{tikzpicture}
    \caption{Illustration for the computation of the Burge code in Proposition~\ref{prop:generic} when $2k>r$ and $\ello'=r-k$.}
    \label{fig:Case2.1}
\end{figure}

\begin{figure}[hbt]
    \centering
    \begin{tikzpicture}[scale=0.75]
\foreach \i in {0} {
    \draw [draw=black] (\i+0,0) rectangle (\i+0.5,6);
    \draw [draw=black, fill=lightgray] (\i+0,0) rectangle (\i+0.5,2);
    
    \draw [->] (\i-0.25,2)--(\i-0.25,0.05);
    \draw node at(\i-0.25,3) [rotate=90] {\tiny{$u-r+k$}} ;
    \draw [->] (\i-0.25,4)--(\i-0.25,5.95);
    
    \draw [draw=black, fill=lightgray] (\i+1,4) rectangle (\i+1.5,6);
    \draw [->] (\i+1.75,4.5)--(\i+1.75,4.05);
    
    \draw node at (\i+1.75,5) [rotate=90]{\tiny{$u-r$}} ;
    \draw [->] (\i+1.75,5.5)--(\i+1.75,5.95); 
    \draw [draw=black] (\i+0.5,3) rectangle (\i+1,6);
    \draw [->] (\i+0.75,4)--(\i+0.75,3.05);
    \draw node at (\i+0.75,4.5) [rotate=90]{\tiny{$r-k$}} ;
    \draw [->] (\i+0.75,5)--(\i+0.75,5.95);
    
    \draw node at (\i-0.25,-.5) {\tiny{$P=(3^{u-r}, 2^{(r-k)-(u-r)}, 1^{u-r+(2k-r)})$}};
    \draw node at (\i+0,-1.5) {\tiny{Burge code:}};
    \draw node at (\i+3,-1.75) {\tiny{$\underbrace{\beta \dots  \beta}_{u-r}$}};

}
  
\foreach \i in {4} {
    \draw [draw=black] (\i+0,2) rectangle (\i+0.5,6);

    \draw [draw=black, fill=lightgray] (\i+0.5,3) rectangle (\i+1,6);
    
   \draw node at (\i+1,-.5) {\tiny{$\partial^{u-r}P=(2^{r-k},1^k)$}};
    
    \draw node at (\i+3,-1.75) {\tiny{$\underbrace{\alpha \dots \dots \alpha}_{r-k}$}};
   }

\foreach \i in {8} {
    \draw [draw=black, fill=lightgray] (\i+0,2) rectangle (\i+0.5,6);
    
    \draw node at (\i+0.75,-.5) {\tiny{$\partial^{u-k}P=(1^k)$}};
    
    \draw node at (\i+2.5,-1.75) {\tiny{$\underbrace{\beta \dots \beta}_{k}$}};  
    
}
\foreach \i in {11}{
    \draw node at (\i+1,-.5) {\tiny{$\partial^{u}P=\varepsilon$}};
    
    \draw node at (\i+2,-1.5) {\tiny{$\alpha$}}; 
    
}
\draw (-3.75,-2.5) rectangle (14,6.5);   
\end{tikzpicture}
    \caption{Illustration for computation of the Burge code in Proposition~\ref{prop:generic}} when $2k>r$ and $\ello'=u-r$.
    \label{fig:Case2.2}
\end{figure}

\begin{proof}[Proof of Theorem~\ref{thm:main}]
    By Proposition~\ref{prop:inclusion} we have that $W^{(u,u-r)}_{k,\ello}\subseteq X^{(u,u-r)}_{k,\ello}$ and by Proposition~\ref{prop:ci} $X^{(u,u-r)}_{k,\ello}$ is a reduced, irreducible complete intersection of codimension $k+\ello-2=\ell(P^{(u,u-r)}_{k,\ello})-2$. Proposition~\ref{prop:generic} shows that the closure of $W^{(u,u-r)}_{k,\ello}$ in $\mathcal N_{J_{(u,u-r)}}$ is $X^{(u,u-r)}_{k,\ello}$ when $\ello = u-r$.

    For $\ello<u-r$ we will use induction on $u-r$. The map $\Phi:\mathcal N_{J_{u+1,u+1-r}}\longrightarrow \mathcal{N}_{J_{(u,u-r)}}$ given by restriction to $\im J_{(u+1,u+1-r)}$ has four-dimensional affine fibers everywhere. By induction, we may assume that the closure of $W^{(u,u-r)}_{k,\ello}$ in $\mathcal N_{J_{(u,u-r)}}$ is $X^{(u,u-r)}_{k,\ello}$ which means that $\dim W^{(u,u-r)}_{k,\ello} = \dim X^{(u,u-r)}_{k,\ello} = 4u-4-3r-(k+\ello-2)$. We have that $\Phi^{-1}\left(W^{(u,u-r)}_{k,\ello}\right) = W^{(u+1,u+1-r)}_{k,\ello}$, which gives that
    $\dim W^{(u+1,u+1-r)}_{k,\ello} = \dim W^{(u,u-r)}_{k,\ello}+4 = 4u-3r-(k+\ello-2) = \dim X^{(u+1,u+1-r)}_{k,\ello}$. Since $X^{(u+1,u+1-r)}_{k,\ello}$ is irreducible, this implies that $X^{(u+1,u+1-r)}_{k,\ello}$ is  the closure of $W^{(u+1,u+1-r)}_{k,\ello}$ in $\mathcal N_{J_{(u+1,u+1-r)}}$.
\end{proof}
Recall that the dominance order between two partitions $P=(p_1,\ldots,p_s), p_1\ge p_2\ge \cdots \ge p_s$, and $P'=(p'_1,\ldots p'_{s'}), p'_1\ge p'_2\ge \cdots \ge p'_{s'}$, of $n=|P|=|P'|$ is $P\ge P'$ if for each $u\in [1,s], \sum_1^up_u\ge \sum_1^up'_u$.\par
In light of the Theorem just proven, we see the following
\begin{cor}\label{dominancecor} For $P_{k,\ello}, P_{k',{\ello}'} $ in $\mathcal T(Q), Q=(u,u-r)$ we have  $P_{k,\ello}\ge P_{k',{\ello}'}$ if one of the following cases hold
\begin{enumerate}
\item $k= k'$ and $ \ello \le {\ello}'$;
\item $k\le k', \ello \le {\ello}'$ and $k'+\ello\le r$.
\end{enumerate}
\end{cor}
\begin{proof}
    This is immediate from the Theorem \ref{thm:main} and the inclusions among the equations of loci of Remark \ref{2.5rem}.  
    \end{proof}
    That $ P^Q_{2,,2}\le P^Q_{1,2}$ in the dominance order for $Q=(5,2)$ (see Example \ref{52ex}) shows that we cannot replace ``if" by ``if and only if'' in Corollary \ref{dominancecor}. \par 
    The proof of the box conjecture for $Q$ with two parts \cite[Theorems 1.2, 3.13]{IKVZ2} listed a specific partition $P_{k,\ello}(Q)$ in $\mathcal T(Q)$: there $\mathcal T(Q)$ is arranged in A rows, and in B and C hooks. In \cite[Theorem~3.19]{IKVZ2} it was shown that the loci $\mathfrak Z^Q_{k,\ello}$ corresponding to these partitions comprise all of $\mathcal N(Q)$, so this set $\mathcal T(Q)$ is all $\mathfrak D^{-1}(Q).$ We show that the partition $P_{k,\ello}(Q)$ in $\mathcal T(Q)$ is the same as the partition $P^Q_{k,\ello}$ that we have described above using the Burge correspondence of \cite{IKM}:
    \begin{proposition}\label{equivprop}
    Let $Q=(u,u-r), r\ge 2.$   The locus $\overline{W^Q_{k,\ello}}$ of $\mathcal N(Q)$ is identical with the analogous locus $\mathfrak Z^Q_{k,\ello}$ of \cite[Theorem 3.13]{IKVZ2}.
        \end{proposition}
    \begin{proof}[Note on proof] The proof depends on showing that the equations defining $\mathfrak Z^Q_{k,\ello}$ are $E^Q_{k,\ello}$ of Definition~\ref{eqndef} above. Since it is quite detailed, and not essential to our main result, we have placed it in the separate note \cite{BIK}.
    \end{proof}

\section{Remarks and Questions}\label{remarksec}
\subsection{Intersections of closures of table loci}\label{intersectsec}
\par
The Corollary \ref{dominancecor} determines the dominance among the loci $W^Q_{k,\ello}$ for stable $Q=(u,u-r), r\ge 2.$   It remains open to determine the intersections $\overline{W^Q}_{k,\ello}\cap \overline{W^Q}_{k',{\ello}'}$ in the absence of inclusion.  These need not be irreducible (Example \ref{52ex}); also for loci $W^Q_P\subset \mathcal N(Q)$ when $P\notin \mathfrak D^{-1}(Q)$ the locus may be reducible.\par
The inclusion
\begin{example}[Closures of the table loci, and their intersections for $Q=(5,2)$.]\label{52ex} Let $Q=(5,2)$, and $B=J_Q$. 
 We may observe from the equations $E(Q)$ for elements of $\mathcal T(Q)$ in Figure \ref{figeqnsloci} that specialization among table loci (i.e. inclusions $W^Q(P_{i,j})\subset \overline{W^Q(P_{k,\ell})}$) is not simply according to increasing row or column index. Let $B=J_Q, Q=(5,2).$ We have $P^Q_{1,1}=Q, P^Q_{1,2}=(5,[2]^2)=(5,1,1), P^Q_{2,1}=(4,[3]^2)=(4,2,1)$ and $P^Q_{2,2}=(4,[3]^3)=(4,1,1,1)$.
The equations in $E(Q)$ show that $\overline{W^Q_{2,1}}\supset \overline{W^Q_{2,2}}$. On the other hand
$\overline{W^Q_{1,2}}$ does not contain $\overline{W^Q_{2,1}}$ or $\overline{W^Q_{2,2}}$, even though $P_{1,2}$ is greater than both $P_{2,1}$ and $P_{2,2}$ in the dominance order! \par
In Figure \ref{figeqnsloci} the displayed entries $W^Q_P$ of $\displaystyle\mathcal W_Q $ include inequalities, and are not just the equations $E^Q_{k,\ello}.$
\par\vskip 0.2cm\par
 {\bf Note.}   We have the following two equalities in $\mathcal W_Q$:
\begin{equation}
{\overline {W^Q_{1,2}}}\cap{\overline {W^Q_{2,1}}}={\overline {W^Q_P}, P=(3,3,1)};
\end{equation}
 \vskip 0.2cm \noindent
 the matrix $A\in\mathcal U_B$ belongs to $ \overline{W^Q_{1,2}}\cap{\overline {W^Q_{2,2}}}$ if and only if $a_1=b_1=g_0h_0=0$. Moreover,
\begin{equation}\label{claim3211eq}
{\overline {W^Q_{1,2}}}\cap{\overline {W^Q_{2,2}}}\subsetneq{\overline {W^Q_P},\, P=(3,2,1,1)}.
\end{equation}
Also, $\overline {W^Q_P},\, P=(3,2,1,1)$ has at least three irreducible components.\par
These are readily shown, see \cite[Example 4.15]{IKVZ1}.
\end{example}\vskip 0.2cm

\begin{figure}[th]
\begin{equation*}
\mathcal T(Q)=\begin{array}{|c|c|}
\hline
(5,2)&(5,[2]^2)\\
\hline
(4,[3]^2)&(4,[3]^3)\\
\hline
\end{array}\, ,\quad 
 \displaystyle\mathcal W_Q =\,\begin{array}{|c|c|}
\hline
a_1b_1\not= 0&b_1=0, a_1\not= 0\\
\hline
a_1=0, g_0b_1h_0\not= 0,&a_1=a_2b_1-g_0h_0=0,\\
a_2b_1-g_0h_0\not=0& a_2\not=0, g_0b_1h_0\not=0.\\
\hline
\end{array}.
\end{equation*}
\caption{Loci $ W^Q_P$, for $P=P_{k,\ello}\in\mathcal{T}(Q), Q=(5,2)$. (Example \ref{52ex}).}\label{figeqnsloci}
\end{figure}

    \vskip 0.2cm
\subsection{Equations for Jordan type loci, for arbitrary stable partitions $Q$}\label{generalsec}\par\vskip 0.2cm
In light of the proof of the Box Theorem (see Theorem~\ref{thm:box}) in \cite{IKM}, we conjecture that Theorem~\ref{thm:main} can be generalized as follows. 

\begin{conjecture}\label{conj}
    Let $Q$ be a stable partition of length $\ell$, and let $P$ be a partition with $\mathfrak D(P) = Q$. Then the closure of $W^Q_P$ in $\mathcal N_{J_Q}$ is an irreducible complete intersection of codimension $\ell(P)-\ell$, with defining equations of degree at most $\ell$.
\end{conjecture}

We first note that while Theorem~\ref{thm:main} establishes the conjecture for $\ell=2$, the conjecture is relatively straightforward for $\ell=1$. Assume that $Q=(u)$ has 1 part. In this case, $\mathcal{N}_{J_Q}$ is isomorphic to $\kk[t]/(t^u)$; indeed each element $B\in \mathcal{N}_{J_Q}$ is a matrix of the following form, which can be identified with $a(t)=\sum a_it^i\in \kk[t]/(t^n)$.
$$B=\begin{bmatrix}
    0&a_1&a_2&\ldots&a_{n-1}\\
    0&0&a_1&\ddots&a_{n-2}\\
    \vdots&\vdots&\ddots&\ddots&\vdots\\
    0&0&0&\ldots&a_1\\
    0&\ldots&\ldots&\ldots&0
\end{bmatrix}$$

Moreover, $\mathfrak{D}(P)=Q$ if and only if $P$ is an almost rectangular partition of $n$, if and only if $P$ is the Jordan type of an element $B\in \mathcal{N}_{J_{(n)}}$. Therefore, for an almost rectangular partition $P$, the closure of $W_P^Q$ in $\mathcal{N}_{J_{(n)}}$ is defined by $\ell(P)-1$ linear equations, namely $a_1=\ldots=a_{\ell(P)-1}=0$. 

The conjecture remains open for $\ell\geq 3$. We believe the inductive approach we used in proving Theorem~\ref{thm:main} can be generalized to prove Conjecture~\ref{conj}. 
The examples we have investigated seem to suggest that the conjecture holds in general but establishing it requires additional steps that go beyond the scope of this paper. Generalizing equations given in Definition~\ref{eqndef} requires more elaborate notations, but will not be fundamentally different from the case $\ell=2$. As for the proof of the conjecture in general, we need a generalized base case for the proof of the analog of Proposition~\ref{prop:inclusion}, as well as a more detailed study of 
the ``outer layer" of the Box associated to $\mathfrak{D}^{-1}(Q)$, in order to generalize Proposition~\ref{prop:inclusion}.

\thanks {\bf Acknowledgment}: 
We thank Bart Van Steirteghem, who participated in our early discussions of equations loci, as well as in \cite{IKVZ1,IKVZ2}; and Rui Zhao, who participated in our \cite{IKVZ1,IKVZ2} and in particular proposed in 2014 the existence of a rectangular table of Jordan types having $Q$ as maximum commuting nilpotent orbit - the box conjecture for $Q$ with two parts, shown in \cite{IKVZ2}. We thank the referee for their comments.  


\begin{thebibliography}{ABDEF}

\bibitem[Ba1]{Ba1}
R. Basili, \emph{On the irreducibility of commuting varieties of nilpotent matrices}, J. Algebra 268 (2003), no. 1, 58--80.
\bibitem[Ba2]{Ba2}
R. Basili, \emph{On the maximal nilpotent orbit that intersects the centralizer of a matrix}, Transform. Groups 27 (2022), no. 1, 1--30. MR 4400714
\bibitem[BaI]{BaI}
R. Basili and A. Iarrobino, \emph{Pairs of commuting nilpotent matrices, and Hilbert function}, J. Algebra 320 (2008), 1235--1254.
\bibitem[BaIK]{BaIK}
R. Basili, A. Iarrobino, and L. Khatami, \emph{Commuting nilpotent matrices and Artinian algebras}, J. Commut. Algebra 2 (2010), no. 3, 
295--325.
\bibitem[BIK]{BIK}
M. Boij, A. Iarrobino, and L. Khatami, \emph{Identifying partitions with maximum commuting orbit $Q=(u,u-r)$}, arXiv:math.AC/2411.18340, (2024).
\bibitem[Br]{Br}
F. Brechenmacher, \emph{ Histoire du th\'{e}or\`{e}me de Jordan de la d\'{e}composition matricielle (1870-1930)}, Th\`{e}se, \'{E}cole des Hautes \'{E}tudes en Sciences sociales, Paris (2006).
\bibitem[Bu1]{Bu1}
W. H. Burge, \emph{A correspondence between partitions related to generalizations of the Rogers-Ramanujan identities}, Discrete Math. 34 (1981), 9-15.

\bibitem[Bu2]{Bu2}
W. H. Burge, \emph{A three-way correspondence between partitions}, Europ. J. Combinatorics 3 (1982), 195-213.

\bibitem[Do]{Do}
R. Docampo, \emph{Arcs on determinantal varieties}, Trans. Amer. Math. Soc. 365, no.5 (2013), 2241-2269.

\bibitem[GI]{GI}
F. Galetto and N. Iammarino, \emph{Computing with jets}, J. of Software for Algebra and Geometry 12 (2022),
43-49.
\bibitem[IK]{IK}
A. Iarrobino and L. Khatami, \emph{Bound on the Jordan type of a generic nilpotent matrix commuting with a given matrix} J. Algebraic Combin. 38 (2013), no. 4, 947--972.

\bibitem[IKVZ1]{IKVZ1}
A. Iarrobino, L. Khatami, B.Van Steirteghem, R. Zhao, \emph{Nilpotent matrices having a given Jordan type as maximum commuting nilpotent orbit}, arXiv:math.RA/1409.2192 v.2, (2015)
\bibitem[IKVZ2]{IKVZ2}
A. Iarrobino, L. Khatami, B.Van Steirteghem, R. Zhao, \emph{Nilpotent matrices having a given Jordan type as maximum commuting nilpotent orbit}, Lin. Alg. Appl. 546 (2018), 210--260.

\bibitem[IKM]{IKM}
J. Irving, T. Ko{\v{s}}ir, and M. Mastnak, \emph{A Proof of the Box Conjecture for Commuting Pairs of Matrices}, arXiv:math.CO/2403.18574 v.2 (2024).
\bibitem[JeSi]{JeSi}
J. Jelisiejew and K. \v{S}ivic, \emph{Components and singularities of Quot schemes and varieties of commuting matrices,}
J. Reine Angew. Math.788 (2022), 129--187.
\bibitem[Jo]{Jo}
C. Jordan, \emph{Trait\'{e} des substitutions et des \'{e}quations alg\'{e}briques}, Gauthier-Villiers, (1870).
\bibitem[Kh1]{Kh1}
L. Khatami, \emph{The poset of the nilpotent commutator of a nilpotent matrix}, Linear Algebra Appl. 439 (12) (2013), 3763--3776.

\bibitem[Kh2]{Kh2}
L. Khatami, \emph{The smallest part of the generic partition of the nilpotent commutator of a nilpotent
matrix}, J. Pure Appl. Algebra 218 (2014), no. 8, 1496--1516.

\bibitem[Kh3]{Kh3}
L. Khatami, \emph{Commuting Jordan types: a survey}, in ``Deformation of Artinian algebras and Jordan type'',
Contemp. Mathematics \#805 (2024), 27-39.
American Mathematical Society, [Providence], RI, ISBN:978-1-4704-7356-3.

\bibitem[KO]{KO}
T. Ko{\v{s}}ir and P. Oblak, \emph{On pairs of commuting nilpotent matrices}, Transform. Groups 14 (2009), no. 1, 175--182.

\bibitem[KSe1]{KS} 
T. Ko{\v{s}}ir and B. A. Sethuraman, \emph{Determinantal varieties over truncated polynomial rings}, J. Pure Appl. Algebra 195 (2005), 75--95.

\bibitem[KSe2]{KS2} 
T. Ko{\v{s}}ir and B. A. Sethuraman, \emph{A Groebner basis for the $2\times 2 $ determinantal ideal mod $t^2$}, J. Algebra 292 (2005), 138--153.

\bibitem[Ma]{Ma}
D. Mallory, \emph{Minimal log discrepencies of determinantal varieties via jet schemes}, J. Pure Appl. Algebra   225 (2021), no. 2, Paper No. 106497, 24 pp.
\bibitem[NeS]{NeS}
M. Neubauer, B.A. Sethuraman. \emph{Commuting Pairs in the Centralizers of 2-Regular Matrices}. J.
Alg. 214, 174-181, 1999
\bibitem[NSi]{NSi}
N. Ngo and K. \v{S}ivic, \emph{ On varieties of commuting nilpotent matrices}, Linear Algebra Appl. 452 (2014), 237--262. MR 3201099

\bibitem[Ob]{Ob}
P. Oblak \emph{On the nilpotent commutator of a nilpotent matrix},
Multilinear Algebra 60 (2012), no. 5, 599--612.

\bibitem[Pan]{Pan}
D. I. Panyushev, \emph{Two results on centralisers of nilpotent elements},
J. Pure and Applied Algebra, 212 no. 4 (2008), 774--779.

\bibitem[Yu]{Yu}
C. Yuen, \emph{Jet schemes of determinantal varieties}, in ``Algebra, geometry and their interactions", 261-270, Contemp. Math. \# 448
American Mathematical Society, Providence, RI, 2007,
ISBN:978-0-8218-4094-8.

\end{thebibliography}
\end{document}